\documentclass[a4, 12pt, 
]{amsart}

\usepackage[all]{xy}
\usepackage{mathrsfs, amsthm, amscd, amsmath, amssymb, graphicx}

\usepackage[dvipdfmx]{hyperref}

\usepackage{setspace}

\newtheorem{theo}{Theorem}[section]

\newtheorem{lemm}[theo]{Lemma}
\newtheorem{cor}[theo]{Corollary}
\newtheorem{claim}[theo]{Claim}

\newtheorem{conj}[theo]{Conjecture}

\numberwithin{equation}{section}

\theoremstyle{definition}
\newtheorem{defi}[theo]{Definition}

\newtheorem{step}{Step}

\theoremstyle{remark}
\newtheorem{rem}[theo]{Remark}

\newcommand{\Spn}[0]{\operatorname{Spn}}

\newcommand{\Image}[0]{\operatorname{Im}}

\newcommand{\End}[0]{\operatorname{End}}

\newcommand{\deldel}{\sqrt{-1}\partial \overline{\partial}}

\newcommand{\e}{\varepsilon}

\newcommand{\reg}{{\rm{reg}}}
\newcommand{\sing}{{\rm{sing}}}

\newcommand{\Cur}[2]{\sqrt{-1}\Theta_{#1}(#2, \bar{#2})}
\newcommand{\Rur}[3]{R_{#1}(#2, \bar{#2}, #3, \bar{#3})}
\newcommand{\met}[3]{#1_{#2 \bar#3}}

\begin{document}

\title[Projective manifolds with semi-positive holomorphic sectional curvature]
{On the image of MRC fibrations\\ of projective manifolds \\
with semi-positive holomorphic sectional curvature}

\author{Shin-ichi MATSUMURA}

\address{Mathematical Institute, Tohoku University, 
6-3, Aramaki Aza-Aoba, Aoba-ku, Sendai 980-8578, Japan.}

\email{{\tt mshinichi-math@tohoku.ac.jp, mshinichi0@gmail.com}}

\date{\today, version 0.01}

\renewcommand{\subjclassname}{%
\textup{2010} Mathematics Subject Classification}
\subjclass[2010]{Primary 32Q10, Secondary 53C25, 14M22.}

\keywords
{Holomorphic sectional curvatures, 
Maximal rationally connected fibrations, 
Rationally connectedness, 
Abelian varieties, 
Ruled surfaces, 
Partially positive curvatures, 
RC positivity, 
Vanishing theorems, 
Minimal models.}

\maketitle

\begin{abstract}
In this paper, we pose several conjectures on 
structures and images of maximal rationally connected fibrations 
of smooth projective varieties admitting semi-positive holomorphic sectional curvature. 
Toward these conjectures, we prove that 
the canonical bundle of images of such fibrations is not big. 
Our proof gives a generalization of 
Yang's solution using RC positivity for Yau's conjecture. 
As an application, we show that 
any compact K\"ahler surface with semi-positive holomorphic sectional curvature 
is rationally connected, 
or a complex torus, 
or a ruled surface over an elliptic curve. 
\end{abstract}

%\tableofcontents

\section{Introduction}\label{Sec-1}

One of the famous conjectures, 
which were posed by S.-T. Yau in \cite{Yau82}, 
states that any compact K\"ahler manifold with negative (resp. positive) holomorphic sectional curvature has an ample canonical bundle (resp. is rationally connected). 

The former conjecture was affirmatively solved for projective varieties 
of dimension $\leq 3$ in \cite{HLW10}, 
solved for projective varieties of arbitrary dimension in \cite{WY16}, 
and solved for compact K\"ahler manifolds in \cite{TY15}. 
On the other hand, 
it is known that a smooth projective variety 
whose holomorphic sectional curvature is identically zero 
admits a finite \'etale cover by an abelian variety 
(see \cite[Proposition 2.2]{HLW16}, \cite{Ber66}, \cite{Igu54}). 
In their paper \cite{HLWZ}, 
Heier-Lu-Wong-Zheng showed that  
any smooth projective variety 
with semi-negative holomorphic sectional curvature admits a finite \'etale cover 
by the product of an abelian variety 
and a projective variety with ample canonical bundle,  
under the assumption of the abundance conjecture (see also \cite{HLW16}). 

The latter conjecture on positive holomorphic sectional curvature 
was affirmatively solved for projective varieties in \cite{HW15} 
and solved for compact K\"ahler manifolds in \cite{Yan17}. 
Therefore one of the remaining most interesting problems in this field 
is to determine a structure of smooth projective varieties 
with \lq \lq semi-positive" holomorphic sectional curvature.

In this paper, 
we pose the following conjecture (Conjecture \ref{str-conj})
by focusing on the similarity to semi-negative holomorphic sectional curvature.  
%\footnote{This conjecture was suggested by Junyan Cao in a private discussion.}
This conjecture also can be seen as a generalization of the structure theorem 
for (holomorphic) bisectional curvature proved in \cite{HSW81} and \cite{Mok88} 
to holomorphic sectional curvature (see also \cite{CG71} and \cite{CG72}). 
As a new approach to rational connectedness, 
Yang  introduced the notation of RC positivity 
in the breakthrough paper \cite{Yan17}. 
Toward Conjecture \ref{str-conj}, 
we study maximal rationally connected (MRC for short) fibrations 
of smooth projective varieties with semi-positive holomorphic sectional curvature, 
by developing the theory of RC \lq \lq semi"-positivity.

\begin{conj}\label{str-conj}
Let $X$ be a compact K\"ahler manifold with semi-positive holomorphic sectional curvature. 
Then there exists a smooth morphism $X \to Y$ such that 
a fiber is rationally connected and 
$Y$ admits a finite \'etale cover $A \to Y$ by an abelian variety $A$. 
%Moreover, the fiber product $\widetilde{X}:=X \times_Y A$ is an 
%isomorphic to the product $A \times R$ of the abelian variety $A$ 
%and a rationally connected projective variety $R$. 
\end{conj}

For a smooth projective variety $X$ with semi-positive holomorphic sectional curvature, 
it seems to be quite difficult to directly confirm that $X$ has positive irregularity 
(in other words, its Albanese map is non-trivial). 
On the other hand, it can be shown that 
a MRC fibration $X \dashrightarrow Y$ of $X$ is non-trivial 
(that is, $0 < \dim Y < \dim X$) 
when $X$ is neither rationally connected nor an abelian variety up to finite \'etale covers.

In this paper, 
we attempt to approach Conjecture \ref{str-conj} 
by investigating a MRC fibration $X \dashrightarrow Y$ of $X$ instead of 
the Albanese map 
(see \cite{Cam92}, \cite{KoMM92} for MRC fibrations and rationally connectedness). 
Note that the image $Y$ of MRC fibrations is determined up to birational equivalence 
(in particular, we may assume that $Y$ is smooth by taking a resolution of singularities),    
and also that the image $Y$ is not uniruled by \cite[Theorem 1.1]{GHS03} 
(equivalently, the canonical bundle $K_Y$ of $Y$ is pseudo-effective by \cite{BDPP}).

From the viewpoint of Conjecture \ref{str-conj}, 
it is natural to expect that a minimal model of the image $Y$ (if exists) 
admits a finite \'etale cover by an abelian variety. 
Further it can also be expected that a MRC fibration of $X$ to a minimal model of $Y$
is actually a smooth morphism. 
For this purpose, it seems to be the first step to show 
that the numerical dimension of the image $Y$ is zero. 
Based on the above observations, we pose the following two conjectures\,$:$

\begin{conj}\label{image-conj}
Let $X$ be a smooth projective variety 
with semi-positive holomorphic sectional curvature, 
and let $X \dashrightarrow Y$  be a dominant rational map 
from $X$ to a smooth projective variety $Y$ 
with the pseudo-effective canonical bundle $K_Y$. 
Then the numerical dimension $\nu(Y)=\nu(K_Y)$ is equal to zero. 
In particular, the numerical dimension of 
the image of non-trivial MRC fibrations of $X$ 
is zero. 
$($See \cite{Nak} for the definition of the numerical dimension $\nu(\cdot)$.$)$
\end{conj}

\begin{conj}\label{mrc-conj}
Let $X$ be a smooth projective variety 
with semi-positive holomorphic sectional curvature, and let 
$X \dashrightarrow Y$ be a MRC fibration of $X$ 
to a projective variety $Y$. 
If $Y$ has $($at most$)$ terminal singularities and 
the canonical divisor $K_Y$ is a nef $\mathbb{Q}$-Cartier divisor 
$($that is, $Y$ is a minimal model$)$, 
then $Y$ is smooth and $f$ is a morphism. 
Moreover $Y$ admits a finite \'etale cover by an abelian variety. 
\end{conj}

In this paper, we prove that the canonical bundle $K_{Y}$ in Conjecture \ref{image-conj}  
is not a big line bundle $($that is, $\nu(Y) < \dim Y$$)$. 
For the proof, we develop the theory of RC positivity introduced in \cite{Yan17} 
(in particular RC semi-positivity).  
(See \cite{Yan18b} and \cite{Yan18c} for the recent development of RC-positivity.)
Our argument in the proof can be seen as a generalization of 
the solution for Yau's conjecture proved in \cite{Yan17}.

\begin{theo}\label{main-thm}
Let $X$ be a compact K\"ahler manifold with semi-positive holomorphic sectional curvature, 
and let $\phi: X \dashrightarrow Y$ be a dominant meromorphic map from $X$ 
to a smooth projective variety $Y$. 
Then $K_{Y}$ is not a big line bundle $($that is, $\nu(Y) < \dim Y$$)$. 
\end{theo}

As a corollary of Theorem \ref{main-thm}, 
we obtain the following result, 
which affirmatively solves Conjecture \ref{mrc-conj} 
for smooth projective surfaces $($even for compact K\"ahler surfaces$)$.

\begin{cor}\label{main-cor}
Let $X$ be a compact K\"ahler surface 
with semi-positive holomorphic sectional curvature. 
Then one of the followings holds\,$:$\vspace{0.1cm}\\
\quad $\bullet$ $X$ is rationally connected. \\
\quad $\bullet$ $X$ is a complex torus. \\
\quad $\bullet$ $X$ is a ruled surface over an elliptic curve. 
\vspace{0.1cm}\\
In particular, Conjecture \ref{mrc-conj} is true 
for compact K\"ahler surfaces with semi-positive holomorphic sectional curvature. 
\end{cor}

In Section \ref{Sec-2}, 
we will recall some results on curvatures of vector bundles and the notion of RC positivity. 
In Section \ref{Sec-3}, we will prove Theorem \ref{main-thm} and its corollary. 

In this paper, 
we interchangeably use the words  
\lq \lq line bundles", \lq \lq invertible sheaves", and \lq \lq Cartier divisors"
(also \lq \lq vector bundles" and \lq \lq locally free sheaves"). 
Further we denote  by the notation $D^{\otimes m}$ 
the $m$-th multiple $mD$ of a divisor $D$. 
Note that we treat only the holomorphic sectional curvature 
obtained from K\"ahler metrics  throughout this paper.

%\begin{rem}\label{main-rem}

%\end{rem}

%\textit{
%Note: 

One year after the previous version of this paper had been put in arXiv, 
Conjecture \ref{str-conj} was affirmatively solved in a \lq \lq strong" form for 
smooth projective varieties in \cite{Mat18a} and \cite{Mat18b}. 
This paper is a modified and shorten manuscript. 
In the previous version, 
it is shown that the numerical dimension of $K_{Y}$ is equal to zero in Theorem \ref{main-thm} under the assumption that $Y$ admits a good minimal model 
(which is true when $\dim Y \leq 3$) 
by developing the techniques in this paper. 
In \cite{Mat18a} and \cite{Mat18b}, 
the strategy explained in this paper turns out well.

 %Further this paper gives a refinement of the solution for Yau's conjecture proved in \cite{Yan17}. 
%Hence the author believes that this paper is worth to display. 
%}

\section{Preliminaries}\label{Sec-2}

\subsection{Curvature of vector bundles}
In this subsection, we fix the notation for various curvatures 
and recall some curvature formulas for induced metrics.

Let $E$ be a (holomorphic) vector bundle of rank $r$ 
on a complex manifold $X$ of dimension $n$ 
and let $g=\sum_{k, \ell}g_{k\bar{l}}\, e_{k}^\vee \otimes \bar e_{\ell}^\vee$ 
be a (smooth) hermitian metric on $E$. 
For the hermitian vector bundle $(E,g)$, 
the Chern curvature 
$$
\sqrt{-1}\Theta_{g}:=\sqrt{-1}\Theta_{(E,g)} 
\in C^{\infty}(X, \Lambda^{1,1}\otimes \End (E))
$$
is defined by 
$$\sqrt{-1}\Theta_{g}
(\partial/\partial z_i, \partial/\partial \bar z_j)(e_k):=
\sqrt{-1}\Big( - 
\frac{\partial^2 \met{g}{k}{\alpha}}{\partial z_i \partial \bar z_j}
g^{\alpha \bar \ell}+
\frac{\partial \met{g}{k}{\alpha}}{\partial z_i} 
g^{\alpha \bar \beta}
\frac{\partial \met{g}{\beta}{\gamma}}{\partial \bar z_j} 
g^{\gamma \bar \ell}
\Big)
e_{\ell}, 
$$
%$$\sqrt{-1}\Theta_{g}
%(\partial/\partial z_i, \partial/\partial \bar z_j)(e_k):=
%\sqrt{-1}\Big( - g^{k \bar \alpha}
%\frac{\partial^2 \met{g}{\alpha}{\ell}}{\partial z_i \partial \bar z_j}+
%g^{k \bar \gamma}\frac{\partial \met{g}{\gamma}{\beta}}{\partial z_i} 
%g^{\beta \bar \alpha}
%\frac{\partial \met{g}{\alpha}{\ell}}{\partial \bar z_j} \Big)
%e_{\ell}, 
%$$
where $(z_1, z_2, \dots, z_{n})$ is a local coordinate of $X$ and 
$\{e_i\}_{i=1}^{r}$ is a local frame of $E$. 
Here we used the Einstein convention for the summation. 

Let $\Lambda^{m}E$ denote the vector bundle defined 
by the $m$-th exterior product of $E$. 
The hermitian metric $g$ on $E$ induces 
the hermitian metric $\Lambda^{m}g$ on $\Lambda^{m} E$. 
It is easy to see that the Chern curvature $\sqrt{-1}\Theta_{\Lambda^{m}g} 
=\sqrt{-1}\Theta_{(\Lambda^{m}E, \Lambda^{m}g)} $ of 
$(\Lambda^{m}E, \Lambda^{m}g)$ satisfies that 
\begin{align}\label{ref1}
\sqrt{-1}\Theta_{\Lambda^{m}g}(v, \bar w) (a_{1}\wedge a_{2} \wedge \cdots \wedge a_{m})
=\sum_{j=1}^{m}(a_{1}\wedge \cdots \wedge \sqrt{-1}\Theta_{g}(v, \bar w) (a_{j}) \wedge 
\cdots \wedge a_{m})
\end{align}
for any tangent vectors $v, w$ in the (holomorphic) tangent bundle $T_X$ 
and any vector $a_{i}$ in $E$. 
Similarly, it can  be seen that 
the hermitian metric $S^{\ell}g$ on the $\ell$-th symmetric product $S^{\ell}E$ of $E$ 
induced by $g$ satisfies that 
\begin{align}\label{ref2}
\sqrt{-1}\Theta_{S^{\ell}g}(v, \bar w) (a_{1}\odot a_{2} \odot \cdots \odot a_{\ell})
=\sum_{j=1}^{\ell}(a_{1}\odot \cdots \odot \sqrt{-1}\Theta_{g}(v, \bar w) (a_{j}) \odot 
\cdots \odot a_{\ell})
\end{align}
for any tangent vectors $v, w \in T_X$ and any vectors $a_{i} \in E$. 
Further, for a hermitian vector bundle $(F,h)$, 
it can also be seen that the induced hermitian metric $g \otimes h$ on $E \otimes F$ 
satisfies that 
\begin{align}\label{ref3}
\sqrt{-1}\Theta_{g \otimes h}(v, \bar w) (a \otimes b)
=\sqrt{-1}\Theta_{g }(v, \bar w) (a)  \otimes b +
 a  \otimes \sqrt{-1}\Theta_{h}(v, \bar w) (b) 
\end{align}
for any tangent vectors $v, w \in T_X$ and any vectors $a \in E$ and $b \in F$.

The curvature tensor 
$$
R_g=R_{(E,g)} \in 
C^{\infty}(X, \Lambda^{1,1}\otimes E^\vee \otimes \bar E^{\vee})
$$
is defined to be 
$$
R_g(v, \bar w,e, \bar f):= \big \langle 
\sqrt{-1}\Theta_{g}(v, \bar w)(e), f \big \rangle_g
$$
for tangent vectors $v, w \in T_X$ and vectors $e, f \in E$. 
Throughout this paper, the notation $E^\vee$ denotes the dual vector bundle of $E$ and 
$\langle \bullet, \bullet \rangle_g$ denotes the inner product with respect to $g$. 
When $E$ is the tangent bundle $T_X$ and $g$ is a hermitian metric on $T_X$, 
the holomorphic sectional curvature $H_g$ is defined to be 
$$
H_g([v]):=\frac{\Rur{g}{v}{v}}{|v|_g^{4}}
$$
for a non-zero tangent vector $v \in T_X$, 
which can be seen as a smooth function 
on the projective space bundle $\mathbb{P}(T_X ^\vee)$ 
(that is, the set of all complex lines $[v]$ in $T_X$). 
The holomorphic section curvature is called {\textit{positive}} 
(resp. {\textit{semi-positive}}) 
if $H_{g}([v]) > 0$  
(resp. $H_{g}([v]) \geq 0$) holds for any non-zero tangent vector $v \in T_X$. 
We remark that there exists the minimum value of $H_{g}$ on $\mathbb{P}(T_{X,p} ^\vee)$ 
at every point $p \in X$ by compactness of  $\mathbb{P}(T_{X,p} ^\vee)$. 

If $g$ is a K\"ahler metric 
(that is, the associated $(1,1)$-form $\omega_g$ is $d$-closed), 
the following symmetry holds: 
$$
R_{g}(e_{i}, \bar e_{j}, e_{k}, \bar e_{\ell})
=R_{g}(e_{k}, \bar e_{\ell}, e_{i}, \bar e_{j})
=R_{g}(e_{k}, \bar e_{j}, e_{i}, \bar e_{\ell}). 
$$
From the above symmetry, we can obtain Royden's lemma (see \cite{Roy80}) 
and a refinement of \cite[Lemma 4.1]{Yan17c}, 
which play a crucial role in the proof of Theorem \ref{main-thm}.

\begin{lemm}[{\cite{Roy80}}]\label{Royden}
Let $g$ be a K\"ahler metric of $X$ and $m$ be an arbitrary positive integer. 
For tangent vectors $\{e_i\}_{i=1}^{m}$ in $T_{X,p}$ at a point $p \in X$, 
the following equality holds\,$:$
\begin{align*}
\sum_{i,j=1}^{m} \Rur{g}{e_i}{e_j}
=\frac{1}{2}\Big\{
\sum_{k=1}^{m} \Rur{g}{e_k}{e_k}+
\frac{1}{4^m}\sum_{\gamma \in I^{m}} \Rur{g}{\eta_{\gamma}}{\eta_{\gamma}} 
\Big\}, 
\end{align*}
where $I:=\{1, -1, \sqrt{-1}, -\sqrt{-1}\}$ and 
$\eta_\gamma :=\sum_{k=1}^{m} \e_{k} e_{k}$
for $\gamma=(\e_1, \e_2, \dots, \e_m) \in I^{m}$. 
In particular, if the holomorphic sectional curvature is semi-positive, 
then the left hand side is non-negative.
\end{lemm}

\begin{lemm}[{\cite[Lemma 4.1]{Yan17c}, \cite[Lemma 6.1]{Yan17}}]\label{Yang-ineq}
Let $g$ be a K\"ahler metric of $X$ and $V$ be a subspace of $T_{X,p}$ at a point $p \in X$. 
If a unit vector $x \in V$ minimizes 
the holomorphic sectional curvature $H_g$ on $V$, 
that it, it satisfies 
$$
\min\{H_g([v])\, |\, {0 \not = v \in V}\}=H_g([x]), 
$$
then we have 
$$
2\Rur{g}{x}{w} \geq (1+|\langle x, w \rangle_g|^2) \Rur{g}{x}{x} 
$$
for any unit vector $w \in V$.  
In particular, if the holomorphic sectional curvature is semi-positive, 
a minimizer $x$ of $H_g$ on $V$ satisfies that 
$$
\Rur{g}{x}{w} \geq 0
$$
for any tangent vector $w \in V$. 
\end{lemm}

The case where $V$ in Lemma \ref{Yang-ineq} coincides with 
the whole tangent space $T_{X,p}$ 
is proved in \cite[Lemma 6.1]{Yan17}. 
It is easy to see that the same argument as in \cite[Lemma 6.1]{Yan17} 
works even in the case of $V$ being a subspace of $T_{X,p}$, 
and thus we omit the proof of Lemma \ref{Yang-ineq}. 
Note that we essentially use the assumption that $g$ is a K\"ahler metric 
in the proof of the above lemmas.

\subsection{RC positivity and vanishing theorems}
In this subsection, 
we recall the notion of RC positivity of vector bundles introduced in \cite{Yan17}. 
Moreover we generalize a vanishing theorem for RC-negative vector bundles 
to treat RC semi-positivity in the proof of Theorem \ref{main-thm}.

\begin{defi}[RC positivity, \cite{Yan17}]\label{RC-def}
A hermitian vector bundle $(E,  g)$ on a complex manifold $X$ is called 
{\it{RC positive}} (resp. {\it{RC negative}}) {\textit{at}} $p \in X$, 
if for any non-zero vector $b \in E_p$ 
there exists a tangent vector $v \in T_{X,p}$ such that 
$$
\Rur{g}{v}{b} > 0\quad \text{(resp. $<0$)} \text{ at } p.
$$
Further $(E,  g)$ is simply called {\it{RC positive}} (resp. {\it{RC negative}}), 
if it is RC positive (resp. RC negative) at every point in $X$. 
\end{defi}

\begin{rem}\label{posi-rem}
A hermitian line bundle is RC positive if and only if it is 
$(n-1)$-positive (that is, it admits a hermitian metric 
whose Chern curvature has at least one positive eigenvalue everywhere).  
Recall that $n$ is the dimension of $X$. 
\end{rem}

If a line bundle admits a hermitian metric 
satisfying the condition of RC positivity (that is, $(n-1)$-positivity), 
then a partial vanishing theorem of Andreotti-Grauert type holds. 
The converse implication (which was first asked in \cite{DPS96}) 
was established in \cite{Yan17b} 
(see also \cite{Mat13}, \cite{Ott12}, \cite{Tot13} for related topics). 
In summary, we have the following result\,$:$

\begin{theo}[\cite{Yan17b}, cf. \cite{DPS96}, \cite{Mat13}]\label{AG-thm}
Let $L$ be a line bundle on a compact complex manifold $X$ of dimension $n$. 
Then the following conditions are equivalent\,$:$\vspace{0.1cm}\\
\quad $\bullet$ The dual line bundle $L^\vee$ is not pseudo-effective. \\
\quad $\bullet$ $L$ admits a hermitian metric 
with RC positive curvature $($$($$n-1$$)$-positive curvature$)$. \vspace{0.1cm}\\
Moreover, when $X$ is a smooth projective variety, 
the above conditions are equivalent to the following condition\,$:$
\vspace{0.1cm}\\
\quad $\bullet$ $L$ is $(n-1)$-ample,  
that is, for any coherent sheaf $\mathcal{F}$ on $X$, 
there is a positive integer $m_0$ such that 
$$
H^{n}(X, \mathcal{F}\otimes L^{\otimes m})=0
\text{ for any } m \geq m_{0}. 
$$
%Here $n$ is the dimension of $X$. 
\end{theo}

In the proof of Theorem \ref{main-thm}, 
we need the following vanishing theorem for partially RC-negative vector bundles, 
which can be seen as a generalization of \cite[Theorem 3.5]{Yan17}.

\begin{theo}\label{vanish-prop}
Let $E$ and $F$ be vector bundles on a compact complex manifold $X$,  
and let $t: F \to E$ be an injective sheaf morphism. 
$($Note that we use the same notation for 
vector bundles and locally free sheaves$)$. 
Assume that there is a $($proper$)$ subvariety $V$ on $X$ 
with the following properties\,$:$ 
\\
%\quad $\bullet$ 
%$t: F \to E$ determines an injective bundle morphism on $X \setminus V$.\\
\quad $\bullet$ 
$E$ admits a hermitian metric $g$ $($defined on $X$$)$ such that  
for any point $p \in X \setminus V$ and 
for any non-zero vector 
$$
b \in \Image (t: F_p \to E_p)  \subset E_p,
$$   
there is a tangent vector $v \in T_{X,p}$ satisfying 
$$
\Rur{g}{v}{b}<0. 
$$

Then we have 
$$
H^{0}(X, F\otimes \mathcal{I}_{V})=0, 
$$
where $\mathcal{I}_V$ is the ideal sheaf associated to the subvariety $V$. 
\end{theo}

\begin{proof}
For a given section $s$ in $H^{0}(X, F\otimes \mathcal{I}_{V})$, 
we consider the section 
$\widetilde{s}$ of $E$ obtained from the induced injective morphism 
$$
H^{0}(X, F\otimes \mathcal{I}_{V}) \subset H^{0}(X, F) \hookrightarrow H^{0}(X, E). 
$$ 
It is sufficient to check that $\widetilde{s}$ is identically zero on $X$. 
We take a point $p_{0} \in X$ that attains the maximum value of 
the (point-wise) norm $|\widetilde{s}|_{g}$. 
We may assume that $\widetilde s(p_0) $ is a non-zero vector in $E$. 
The section $\widetilde{s}$ is identically zero on $V$ 
by the construction of $\widetilde{s}$, 
and thus $p_{0}$ is outside the subvariety $V$. 

Now we have the following equality\,$:$
\begin{align}\label{max}
\deldel |\widetilde{s}|^2_{g}
= \sqrt{-1}\big \langle D'\widetilde{s}, D'\widetilde{s} \big \rangle_g - 
\big \langle \sqrt{-1}\Theta_{g}
(\widetilde{s}), \widetilde{s} \big \rangle_g,  
\end{align}
where $D'$ is the $(1,0)$-part of the Chern connection. 
The left hand side is a semi-negative $(1,1)$-form at $p_{0}$ by the choice of $p_0$. 
%By applying the non-zero vector $\widetilde{s}(p_0) \in E_{p_0}$, 
On the other hand, since 
$$
b:=\widetilde{s}(p_0) \in \Image (t: F_{p_0} \to E_{p_0})  \subset E_{p_0}
$$
is a non-zero vector, we can find a tangent vector $v \in T_{X,p_0}$ such that 
$$
\big \langle \sqrt{-1}\Theta_{g}(v, \bar{v})
(\widetilde{s}), \widetilde{s} \big \rangle_g = 
\Rur{g}{v}{b}<0
\text{ at $p_{0}$}
$$
by the assumption. This is a contradiction.
\end{proof}

\section{Proof of Theorem \ref{main-thm} and its corollaries}\label{Sec-3}

This section is devoted to the proof of Theorem \ref{main-thm} and its corollary.

\begin{proof}[Proof of Theorem \ref{main-thm}]
The proof can be divided into two steps. 
The main idea comes from Step \ref{step1}, 
in which we consider the situation of $\phi$ being a smooth morphism. 
In Step \ref{step2}, by modifying this idea in Step \ref{step1} 
to treat an arbitrary meromorphic map $\phi$, 
we prove a generalization of Theorem \ref{main-thm} 
for a projective variety $Y$ with canonical singularities.

%Finally we prove the general case in Step \ref{step4} 
%by investigating a minimal model of $Y$ and its canonical model. 

\begin{step}[The case of $\phi$ being a smooth morphism]\label{step1}
In this step, we show only that $K_Y$ is not an ample line bundle 
under the situation 
that $\phi: X \dashrightarrow Y$ in Theorem \ref{main-thm} 
is a smooth morphism.

Let $g$ be a K\"ahler metric of $X$ with semi-positive holomorphic sectional curvature. 
The surjective bundle morphism 
$$
(T_X, g) \xrightarrow{\quad d\phi_* \quad} (\phi^{*} T_Y, h)
$$
can be obtained from the differential map $d\phi_*$. 
We remark that the above morphism is surjective as a bundle morphism 
since $\phi$ is a smooth morphism. 
Further the K\"ahler metric $g$ and the above bundle morphism 
induce the hermitian metric $h$ on the pull-back $\phi^{*} T_Y$ of 
the tangent bundle $T_Y$ of $Y$. 
We put $m := \dim Y$, $\Omega_X:=T_X^{\vee}$, and $\Omega_Y:=T_Y^{\vee}$. 
We obtain the injective bundle morphism 
\begin{equation}\label{eq-surj}
(\phi^{*} \Lambda^m \Omega_Y=\phi^{*} K_Y, \Lambda^m h^\vee)
\xrightarrow{\quad \Lambda^m d\phi^* \quad} (\Lambda^m \Omega_X, \Lambda^m g^\vee)
\end{equation}
by taking the dual vector bundle and the $m$-th exterior product, 
where $\bullet^\vee$ denotes the dual bundle of vector bundles  
or the dual hermitian metric. 
Then the following claim follows from Royden's lemma. 

\begin{claim}\label{claim-1}
For any point $p \in X$ and any non-zero vector 
$$
b \in \Image \big( \Lambda^m d\phi^* :\phi^{*} K_{Y,p} 
\xrightarrow{ \quad \Lambda^m d\phi^* \quad } \Lambda^m \Omega_{X,p} \big)
\subset \Lambda^m \Omega_{X,p} \text{ at } p, 
$$
there exists a tangent vector $v \in T_{X,p}$ 
with the following property\,$:$
$$\text{
$\bullet$ $d\phi_{*}(v) \not = 0$ in $T_{Y, \phi(p)}$. \quad \quad 
$\bullet$ $\Rur{\Lambda^m g^\vee}{v}{b} \leq 0$. 
} 
$$
\end{claim}

\begin{rem}\label{smooth-rem}
Even if $\phi$ is not a smooth morphism on the whole space $X$,  
the argument below still works for a point $p \in X$ at which $\phi$ is smooth. 
This argument will be used again in Step \ref{step2}. 

\end{rem}

\begin{proof}[Proof of Claim \ref{claim-1}]
For a given point $p \in X$, 
we choose an orthonormal basis $\{e_{i}\}_{i=1}^{n}$ of the tangent space 
$T_{X, p}$ at $p$ 
such that $ \{d\phi_{*}(e_{i})\}_{i=1}^{m}$ is also 
an orthonormal basis of $\phi^{*} T_{Y, p}=T_{Y,\phi(p)}$.  
Here $n$ denotes the dimension of $X$. 
We define $V$ by the subspace 
$$
V:=\Spn \langle \{ \{e_i\}_{i=1}^{m} \} \rangle \subset T_{X,p}
$$
spanned by $\{e_i\}_{i=1}^{m}$ and 
the vector $a$ by 
$$
a:=e_{1} \wedge e_{2} \wedge \cdots \wedge e_{m} \in 
\Lambda^m V \subset 
\Lambda^m T_{X,p}.  
$$
It is sufficient for the proof to find a tangent vector $v \in T_{X,p}$ 
such that 
$$\text{
$d \phi_{*}(v) \not = 0$ in $T_{Y, \phi(p)}$ 
\quad and \quad 
$\Rur{\Lambda^m g}{v}{a} \geq 0$  
} 
$$
for a non-zero vector $a=e_{1}\wedge e_2 \wedge \cdots \wedge e_m \in \Lambda ^{m} V$, 
since the image $\Image (\Lambda^m d\phi^*) \subset \Lambda^m \Omega_{X,p}$ 
is spanned by the vector 
$e^\vee_{1} \wedge e^\vee_{2} \wedge \cdots \wedge e^\vee_{m}$, 
where $\{e^\vee_{i}\}_{i=1}^n$ denotes the dual basis of 
$\{e_i\}_{i=1}^n$.

For an arbitrary index $i \in \{1,2,\dots, m\}$, 
we put 
$$
A_{i}:= \Rur{\Lambda^{m}g}{e_i}{a},   
$$ 
%Then we can obtain 
%\begin{align*}
%A_{i}
%=\Rur{\Lambda ^{m}g}{e_i}{a}
%\leq \Rur{\Lambda^{m}h}{e_i}{a'}
%=\Cur{\Lambda^{m}h}{e_i}|a'|^2_{\Lambda^{m}h}
%=\Cur{\Lambda^{m}h}{e_i}. 
%\end{align*}
and for simplicity we put 
$$
B_i(\bullet):=\Cur{g}{e_{i}}(\bullet) \in \End (T_{X,p}).   
$$ 
Then we can easily check that 
$$
\Cur{\Lambda ^{m}g}{e_i}(a)=\sum_{j=1}^{m} 
e_{1}\wedge \cdots \wedge B_{i}(e_j) \wedge 
\cdots \wedge e_{m}
$$
by the definition of $\sqrt{-1}\Theta_{\Lambda ^{m}g}$ 
(see equality (\ref{ref1})). 
A straightforward computation yields 
\begin{align}\label{ref4}
A_{i}&=\big\langle \Cur{\Lambda ^{m}g}{e_i}(a), a \big\rangle_{\Lambda ^{m}g}\\ \notag
&=\sum_{j=1}^{m} \big\langle e_{1}\wedge \cdots \wedge B_{i}(e_j) \wedge 
\cdots \wedge e_{m}, e_{1}\wedge e_2 \wedge \cdots \wedge e_{m} 
\big\rangle_{\Lambda ^{m}g}\\  \notag
&=\sum_{j=1}^{m} \big\langle B_{i}(e_j), e_j \big\rangle_{g}\\  \notag
&=\sum_{j=1}^{m} \Rur{g}{e_i}{e_j}. 
\end{align}
By Royden's lemma (see Lemma \ref{Royden}) and 
the assumption of the holomorphic sectional curvature being semi-positive, 
we can obtain   
\begin{align*}
\sum_{i=1}^{m}A_{i} = \sum_{i,j=1}^{m} \Rur{g}{e_i}{e_j}
=\frac{1}{2}\Big\{
\sum_{k=1}^{m} \Rur{g}{e_k}{e_k}+
\frac{1}{4^m}\sum_{\gamma \in I^{m}} \Rur{g}{\eta_{\gamma}}{\eta_{\gamma}} 
\Big\}\geq 0. 
\end{align*}
Therefore it can be seen that 
$$
A_{i_0}= \Rur{\Lambda^{m}g}{e_{i_0}}{a} \geq 0$$ 
for some $i_{0}\in \{1,2,\dots, m\}$. 
By the choice of the orthonormal basis, 
the vector $d\phi_{*}(e_{i_{0}})$ is a non-zero vector in $T_{Y, \phi(p)}$. 
This completes the proof. 
\end{proof}

In the rest of this step, 
we show that $K_{Y}$ is not an ample line bundle 
by using Claim \ref{claim-1} and Theorem \ref{AG-thm}. 
Since $(\phi^{*} K_Y, \Lambda^m h^\vee)$ 
is a subbundle of $(\Lambda^m \Omega_X, \Lambda^m g^\vee)$ 
as hermitian vector bundles, 
we have
$$
\Cur{\Lambda^{m} h^\vee}{v}|c|^2_{\Lambda^{m} h^\vee}
=\Rur{\Lambda^{m} h^\vee}{v}{c} \leq 
\Rur{\Lambda^{m} g^\vee}{v}{b} 
$$
for any vector $v \in T_X$ and a non-zero vector $c \in \phi^{*}K_Y$, 
where $b:=\Lambda^{m} d\phi^{*}(c) \in \Lambda^{m} \Omega_X$. 
For an arbitrary point $p \in X$,  
we can find a tangent vector $v \in T_{X,p}$ 
such that 
$$\text{
$d\phi_{*}(v) \not = 0$ in $T_{Y, \phi(p)}$ \quad and \quad 
$\Cur{\Lambda^{m} h^\vee}{v} \leq 0$
} 
$$
by Claim \ref{claim-1} and $|c|_{\Lambda^{m} h^\vee} \not =0$. 
 
On the other hand, 
if $K_Y$ is assumed to be an ample line bundle, 
there is a smooth hermitian metric $H$ on $K_Y$ with 
the (strictly) positive curvature $\sqrt{-1}\Theta_{H}>0$. 
Then it can be shown that the line bundle $\phi^{*}K_Y$ is RC negative. 
Indeed, 
for the hermitian metric on $\phi^{*}K_Y$
defined by 
$$
(\Lambda^{m} h^\vee)^2 \cdot (\phi^{*}H)^{-1}, 
$$
the curvature 
$\sqrt{-1}\Theta_{(\Lambda^{m} h^\vee)^2 \cdot (\phi^{*}H)^{-1}}$ 
satisfies that 
\begin{align*}
\sqrt{-1}\Theta_{(\Lambda^{m} h^\vee)^2 \cdot (\phi^{*}H)^{-1}} (v, \bar v)
&=2\Cur{(\Lambda^{m} h^\vee)}{v} -\phi^{*}\Cur{H}{v} \\
&\leq -\Cur{H}{d\phi_*(v)}\\
&<0 
\end{align*}
for the tangent vector $v \in T_{X,p}$ obtained in Claim \ref{claim-1}. 
The last inequality follows from $d\phi_{*}(v) \not = 0$ 
and $\sqrt{-1}\Theta_H>0$. 

The dual bundle $\phi^{*}K_Y^{\vee}$ is RC positive, 
and thus $\phi^{*}K_Y$ is not pseudo-effective by Theorem \ref{AG-thm}. 
It contradicts to the assumption that $K_Y$ is an ample line bundle. 
\end{step}

\begin{step}[The proof of Theorem \ref{main-thm}]\label{step2}
In this step, we prove the following statement by modifying the idea in Step \ref{step1}, 
which is a generalization of  Theorem \ref{main-thm}. 

\begin{theo}\label{main-thm2}
Let $X$ be a compact K\"ahler manifold with semi-positive holomorphic sectional curvature, 
and let $\phi: X \dashrightarrow Y$ be a dominant meromorphic map from $X$ 
to a projective variety $Y$ with at most canonical singularities 
Then $K_{Y}$ is not a  big line bundle. 
\end{theo}

\begin{proof}[Proof of Theorem \ref{main-thm2}]

Even if $\phi$ is a morphism, 
the morphism (\ref{eq-surj}) is not injective as a bundle morphism 
(since the rank of the linear map defined on fibers may not be constant), 
but it induces the injective sheaf morphism 
between locally free sheaves $\phi^{*} K_Y$ and $\Lambda^m \Omega_X$. 
From now on, we interchangeably use the words 
\lq \lq vector bundles" and \lq \lq locally free sheaves",  
and we use the same notation for the induced sheaf morphism. 
The main differences from Step \ref{step1} are 
that we have to treat the indeterminacy locus and 
that we can not obtain 
the induced metric on $\phi^{*} K_Y$ 
since the morphism (\ref{eq-surj}) is not a bundle morphism. 
To overcome these difficulties, 
we apply Theorem \ref{vanish-prop} instead of Theorem \ref{AG-thm}.

We first take a resolution $\tau: \overline{X} \to X$ of 
the indeterminacy locus $B$  of $\phi$ 
such that it passes through a resolution $\mu : \overline{Y} \to Y$  of singularities of $Y$.  
%We may assume that $\mu$ (resp. $\tau $) is an isomorphism over 
%the non-singular locus $Y_{\reg}:=Y\setminus Y_{\sing}$ 
%(resp. on $\bar \phi^{-1}(Y_{\reg}) \cap \tau^{-1}(X \setminus B)$). 
The morphisms $\varphi$ and $\bar \phi$ 
are defined by the following diagram\,$:$
\begin{equation*}
\xymatrix@C=40pt@R=30pt{
\overline{X} \ar[d]_\tau \ar[rd]^{\bar{\phi}\ \ } \ar[r]^{\varphi}  & \overline{Y}  \ar[d]^{\mu} \\ 
X \ar@{-->}[r]^{\phi \ \ \ }  &  Y.\\   
}
\end{equation*}

For a contradiction, we assume that $K_Y$ is a big line bundle. 
It can be seen that there exist a very ample line bundle $A$ on $Y$ 
and an effective Cartier divisor $E$ on $Y$
such that $K_Y^{\otimes m_{0}}=A \otimes E$ holds for some $m_{0}>0$
by Kodaira lemma. 
We define the \lq \lq pull-backs" of 
the Cartier divisors $K_Y^{\otimes m_{0}}$ and $A$ 
by  
$$
\phi^{*}K_Y^{\otimes m_{0}}
:= \tau_{*} \bar \phi^{*}(K_Y^{\otimes m_{0}}) \quad \text{and} \quad 
\phi^{*} A=\tau_{*} \bar \phi^{*}A. 
$$

Let $\{t_{i}\}_{i \in I}$ be a basis of $H^{0}(Y, A)$. 
The sections $\{t_{i}\}_{i \in I}$ determine  
the smooth hermitian metric $H$ on $A$. 
Indeed, the hermitian metric $H$ on $A$ can be defined to be 
$$
|e|^2_H:=
\frac{|e|^2} {\sum_{i\in I} |t_i|^2} 
$$
for every vector $e \in A$. 
It follows that 
the Chern curvature $\sqrt{-1}\Theta_{H}$ is a positive $(1,1)$-form 
on the non-singular locus $Y_{\reg}:=Y \setminus Y_{\sing}$ of $Y$ 
since $A$ is a very ample line bundle on $Y$. 

Similarly, the pull-backs 
$\{\phi^{*}t_{i}\}_{i \in I}$ of the sections $\{t_{i}\}_{i \in I} $ under $\phi$, 
which are sections of $\phi^* A$, 
also determine the \lq \lq singular" hermitian metric on $\phi^{*}A$, 
which we denote by the notation $\phi^* H$ 
(see \cite{Dem} for singular hermitian metrics). 
The section $\phi^{*}t_{i}$ obtained from the pull-back of $t_i$ 
is identically zero on the indeterminacy locus $B$  
(otherwise it contradicts to the fact that 
$B$ is the indeterminacy locus and $A$ is very ample). 
Hence we can see that $\phi^* H$ has analytic singularities along the indeterminacy locus $B$.

We consider a point $p \in X$ such that 
$\phi(p) \in Y_{\reg}$ and $\phi$ is a morphism at $p$. 
%For a positive integer $\ell$, 
%we have the injective sheaf morphism 
%$$
%S^{\ell} (\Lambda^m d\phi^*) : 
%\phi^{*} K_{Y}^{\otimes \ell}
%\xrightarrow{\quad   \quad }
%S^{\ell}(\Lambda^m \Omega_{X}) 
%$$
%on a neighborhood of $p$. 
%It can be seen that $\phi$, $\bar \phi$, and $\varphi$ are identified 
%by $\tau $ and $\mu$ on a neighborhood of $p$.
It can be seen that 
$\phi^{*}H$ is smooth at $p$ and 
that $\sqrt{-1}\Theta_{\phi^*  H}=\phi^*  \sqrt{-1}\Theta_{H}$ holds at $p$.  
Therefore, by Claim \ref{claim-1} (see also Remark \ref{smooth-rem}) and Step \ref{step1}, 
we can obtain the following claim\,$:$

%Now we have the injective sheaf morphism 
%\begin{equation}\label{eq-surj2}
%\phi^{*} K_Y^{\otimes \ell-1}=\phi^{*} K_Y^{\otimes \ell} \otimes \phi^{*}K_{Y}^{\vee}
%\xrightarrow{\quad S^{\ell}(\Lambda^m d\phi^*) \otimes id_{\phi^{*} K_{Y}^{\vee}} \quad } 
%S^{\ell}(\Lambda^m \Omega_X) \otimes \phi^{*} K_{Y}^{\vee}
%\end{equation}
%for every positive integer $\ell$. 
%The vector bundle $S^{\ell}(\Lambda^m \Omega_X) \otimes K_{Y}^{\vee}$ 
%is equipped with the hermitian metric $S^{\ell}(\Lambda^m g^\vee) \otimes \phi^{*}H^{\vee}$, 
%where $H$ is a smooth hermitian metric on $K_Y$ with 
%the positive curvature $\sqrt{-1}\Theta_{H}>0$. 

\begin{claim}\label{claim-2}
We consider a point $p \in X$ such that 
$\phi(p) \in Y_{\reg}$ and $\phi$ is a morphism at $p$. 
Let $\ell$ be a positive integer. 
Then, for any non-zero vector 
$$
b \in \Image \big( S^{\ell} (\Lambda^m d\phi^*) : 
\phi^{*} K_{Y,p}^{\otimes \ell}
\xrightarrow{\quad S^{\ell}(\Lambda^m d\phi^*)  \quad }
S^{\ell}(\Lambda^m \Omega_{X,p}) 
\big) \subset S^{\ell}(\Lambda^m \Omega_{X,p}) \text{ at } p, $$ 
there exists a tangent vector $v \in T_{X,p}$ 
with the following property\,$:$
$$\text{
$\bullet$ $d\phi_{*}(v) \not = 0$ in $T_{Y, \phi(p)}$. \quad \quad 
$\bullet$  $\Rur{S^{\ell}(\Lambda^m g^\vee)}{v}{b} \leq 0$. 
} 
$$
\vspace{-0.5cm}\\
Moreover, for such a point $p$ and a non-zero vector 
$$
b \in \Image \big( S^{\ell} (\Lambda^m d\phi^*)\otimes id : 
\phi^{*} K_{Y,p}^{\otimes \ell} \otimes \phi^* A_p^{\vee}
\xrightarrow{\quad S^{\ell}(\Lambda^m d\phi^*) \otimes id \quad }
S^{\ell}(\Lambda^m \Omega_{X,p}) \otimes \phi^* A_p^{\vee}
\big) \text{ at } p, $$ 
there exists a tangent vector $v \in T_{X,p}$ such that 
$$\Rur{S^{\ell}(\Lambda^m g^\vee) \otimes \phi^* H^\vee}{v}{b} < 0. 
$$
\end{claim}

\begin{proof}[Proof of Claim \ref{claim-2}]
We choose an orthonormal basis $\{e_{i}\}_{i=1}^{n}$ of 
$T_{X, p}$ at $p$ such that $ \{d\phi_{*}(e_{i})\}_{i=1}^{m}$ is also 
an orthonormal basis of $\phi^* T_{Y,p}$. 
Note that the morphism $\phi$ is a smooth morphism at $p$ 
(otherwise there is no non-zero vector in the image). 
Let $V  \subset T_{X,p}$ be the subspace 
$V:=\Spn \langle \{ \{e_i\}_{i=1}^{m} \} \rangle $
spanned by $\{e_i\}_{i=1}^{m}$ and $a^{\odot \ell}$ be 
the vector defined by 
$$
a^{\odot \ell}:=(e_{1} \wedge e_{2} \wedge \cdots \wedge e_{m})^{\odot \ell} 
\in S^{\ell} (\Lambda^{m} V) \subset S^{\ell}(\Lambda^m T_{X,p}). 
$$
For any $i \in \{1,2,\dots, m\}$, we obtain  
\begin{align*}
\Cur{S^{\ell}(\Lambda ^{m}g)}{e_i}(a^{\odot \ell})&=
\sum_{k=1}^{\ell} a \odot \cdots \odot 
\Cur{\Lambda ^{m}g}{e_i}(a) \odot \cdots \odot a \\
&=\ell \Cur{\Lambda ^{m}g}{e_i}(a) \odot a^{\odot \ell-1}  
\end{align*}
from equality  (\ref{ref2}). Hence it can be shown that     
\begin{align}\label{ref5}
\Rur{S^{\ell}(\Lambda^{m}g)}{e_i}{a^{\odot \ell}}
&=\big\langle \Cur{S^{\ell}(\Lambda ^{m}g)}{e_i}(a^{\odot \ell}), 
a^{\odot \ell} \big\rangle_{S^{\ell}(\Lambda ^{m}g)}\\ \notag
&=\ell \big\langle 
\Cur{\Lambda ^{m}g}{e_i}(a) \odot a^{\odot \ell-1}, 
a^{\odot \ell} \big\rangle_{S^{\ell}(\Lambda ^{m}g)}\\ \notag
&=\ell \big\langle \Cur{\Lambda ^{m}g}{e_i}(a), 
a \big\rangle_{\Lambda ^{m}g}. 
\end{align}
By Royden's lemma (see Lemma \ref{Royden}) and the proof of Claim \ref{claim-1}, 
we can easily check that the right hand side is non-negative  
for some $i_0 \in \{1,2,\dots, m\} $. 
This leads to the first conclusion. 

We will check the latter conclusion. 
The vector $b$ in the claim can be written as 
$b=b_1 ^{\odot \ell} \otimes b_2$, 
where $b_{1}$ is a vector in the image of 
$$
\phi^{*} K_{{Y,p}}
\xrightarrow{\quad \Lambda^md  \phi^* \quad} 
\Lambda^m \Omega_{X,p}, 
$$
and $b_{2}$ is a vector in $\phi^{*} A_p^{\vee}$. 
Then, for any tangent vector $v \in T_{X,p}$, 
we obtain 
\begin{align*}
&\Rur{S^{\ell}(\Lambda^m g^\vee) \otimes \phi^* H^\vee}{v}{b}\\
=&\Rur{S^{\ell}(\Lambda^m g^\vee)}{v}{b_1^{\odot \ell}}|b_2|^2_{\phi^* H^\vee}+
|b_1^{\odot \ell}|^2_{S^{\ell}(\Lambda^m g^\vee)}
\Rur{\phi^* H^\vee}{v}{b_2}\\
=&\Rur{S^{\ell}(\Lambda^m g^\vee)}{v}{b_1^{\odot \ell}}|b_2|^2_{\phi^* H^\vee}
+
|b_1^{\odot \ell}|^2_{S^{\ell}(\Lambda^m g^\vee)}
|b_2|^2_{\phi^* H^\vee}
\sqrt{-1}\Theta_{\phi^* H^\vee}(v, \bar v)
\end{align*}
from (\ref{ref3}) and (\ref{ref5}). 
When the tangent vector $v$ satisfies the first conclusion, 
we can see that 
$$
\Rur{S^{\ell}(\Lambda^m g^\vee)}{v}{b_1^{\odot \ell}} 
\leq 0 \quad \text{ and } \quad 
\Cur{\phi^{*}H^\vee}{v}=-\Cur{H}{d\phi_{*}(v)}<0
$$
from $d\phi_{*}(v) \not =0$ and $\sqrt{-1}\Theta_H>0$.
This completes the proof.
\end{proof}

%For a (non-empty) Zariski open set $Y_0$ in $Y$ over which $\phi$ is a smooth morphism, 
%we consider the smooth morphism 
%$\phi_0:=\phi|_{X_{0}}$ restricted to $X_{0}:=\phi^{-1}(Y_{0})$. 
%We have the following diagram\,$:$
%\begin{equation*}%\ar@{^{(}->}
%\xymatrix{
%X \ar[d]^\phi& \ar@{_{(}->}[l] X_0:=\phi(Y_0) \ar[d]^{\phi_{0}:=\phi|_{X_0}} \\ 
%Y & \ar@{_{(}->}[l] Y_0.\\   
%}
%\end{equation*}

In the rest of this step, 
we will finish the proof of Theorem \ref{main-thm2}
by applying  the above claim and Theorem \ref{vanish-prop}. 
For a sufficiently divisible integer $\ell=k m_0$, 
we consider the formula 
$$
\mu^{*} K_{Y}^{\otimes \ell} = K_{\overline{Y}}^{\otimes \ell}\otimes F^{\otimes -\ell}. 
$$ 
Here $F$ is the effective divisor 
since $Y$ has at most canonical singularities. 
Then we obtain the injective sheaf morphisms 
\begin{align}\label{eq-surj3}
\bar \phi^{*} A^{\otimes k-1} 
&\xrightarrow {\quad \otimes t_{0} \quad }
\bar \phi^{*} A^{\otimes k-1} \otimes \bar \phi^{*}E^{\otimes k}
=\bar \phi^{*} K_Y^{\otimes \ell} \otimes  \bar \phi^{*} A^{\vee}
= \varphi^{*} (K_{\overline{Y}}^{\otimes \ell}\otimes F^{\otimes -\ell}) 
\otimes \bar \phi^{*}A^{\vee}\\
&\xrightarrow {\quad \otimes t \quad } \notag
\varphi^{*} K_{\overline{Y}}^{\otimes \ell} \otimes \bar \phi^{*} A^{\vee}
\xrightarrow {\quad S^{\ell}(\Lambda^m d \varphi^{*}) \otimes id \quad }
S^{\ell}(\Lambda^m \Omega_{\overline{X}}) \otimes \bar \phi^{*} A^{\vee}, 
\end{align}
where $\otimes t$ (resp. $t_0$)
is the multiplication map defined by 
the natural section $t$ (resp. $t_0$) of the effective divisor 
$\varphi^* F^{\otimes \ell}=\ell \varphi^*F$ 
(resp. $\bar \phi^{*} E^{\otimes k}=k\bar \phi^{*}E$). 
Further we have 
\begin{align}\label{push}
\tau_{*}\big( 
S^{\ell}(\Lambda^m \Omega_{\overline{X}}) 
\otimes \bar \phi^{*} A^{\vee} \big)
 = S^{\ell}(\Lambda^m \Omega_{X}) \otimes \phi^{*} A^{\vee} \text { and }
\tau_* (\bar\phi^{*} A^{\otimes k-1}) = \phi^{*} A^{\otimes k-1}
\end{align}
by the definition. 
Therefore we obtain the injective sheaf morphism
\begin{align}\label{inj}
\phi^{*} A^{\otimes k-1} \longrightarrow
 S^{\ell}(\Lambda^m \Omega_{X}) \otimes \phi^{*} A^{\vee}. 
\end{align}
By taking the pull-back under $\bar\phi $,  
chasing the injective morphisms induced by (\ref{eq-surj3}), 
and using equality (\ref{push}), 
we obtain the following diagram\,$:$
\begin{equation*}
\xymatrix@C=40pt@R=30pt{
H^{0}(Y,A^{\otimes k -1})  \ar[r]^{\phi^*}  \ar@/_20pt/[dr]^{\bar \phi^*}
& H^{0}(X, \phi^{*} A^{\otimes k-1}) 
\ar[r]^{(\ref{inj})\ \ \ \ \ \ \ \ \ } \ar[d]_{\cong} 
& H^{0}(X, S^{\ell}(\Lambda^m \Omega_{X}) \otimes \phi^{*} A^{\vee}) \ar[d]_{\cong} \\ 
& H^{0}(\overline{X}, \bar \phi^{*} A^{\otimes k-1} ) 
\ar[r]^{(\ref{eq-surj3})\ \ \ \ \ \ \ \ \ }
& H^{0}(\overline{X}, S^{\ell}(\Lambda^m \Omega_{\overline{X}}) \otimes \bar \phi^{*} A^{\vee}).\\   
}
\end{equation*}

By taking a sufficiently large integer $k$, 
we can choose a non-zero section $s$ in $H^{0}(Y, A^{\otimes k-1})$ 
such that $s$ is identically zero on the singular locus $Y_{\sing}$, 
by ampleness of $A$. 
We consider the non-zero section 
$$
\widetilde s \in H^{0}(X, S^{\ell}(\Lambda^m \Omega_{X}) \otimes \phi^{*} A^{\vee})  
$$ 
obtained from the above injective morphisms.  
The metric $\phi^{*}H$ is a singular hermitian metric, 
but it has analytic singularities, 
and thus $\phi^{*}H^\vee$ can be seen  locally as a smooth function 
(which is identically zero on $B$). 
Therefore 
the point-wise norm  
$|\widetilde {s}|_{S^{\ell}(\Lambda^m g^\vee ) 
\otimes \phi^{*}H^{\vee}} $
of $\widetilde s$
is a smooth function on $X$. 
Thus we can take a maximizer $p_0 \in X$ of this norm, 
that is, $p_0 \in X$ satisfies that 
$$
\max_X |\widetilde {s}|_{S^{\ell}(\Lambda^m g^\vee ) 
\otimes \phi^{*}H^{\vee}} = 
|\widetilde {s}|_{S^{\ell}(\Lambda^m g^\vee)  
\otimes \phi^{*}H^{\vee}}(p_0). 
$$
It can be seen that that $\widetilde{s}$ is identically zero on $B$ 
since $\widetilde{s}$ is obtained via the pull-back under $\phi$. 
In particular, the point $p_0$ is outside $B$. 
%Therefore $\phi$ is a morphism at $p_0$, 
%and thus $\widetilde{s}$ is just the pull-back under the morphism $\phi$
%on a neighborhood of $p_0$. 
Further it follows that 
$\widetilde{s}$ is identically zero over $Y_{\sing}$ 
by the choice of $s$. 
Therefore we can easily see that 
the same argument as in Theorem \ref{vanish-prop} works. 
(The only difference is that $\phi^{*}H$ is a singular hermitian metric, 
but it is smooth on a neighborhood of $p_{0}$.) 
Indeed, by applying equality (\ref{max}) 
to the non-zero vector $b:= \widetilde{s}(p_{0})$, 
we can conclude that $\widetilde{s}$ is identically zero 
thanks to Claim \ref{claim-2}. 
This is a contradiction. 
\end{proof}

%This technique will be used in Step \ref{step4}. 
\end{step}

Compared to Step \ref{step1}, 
the difficulty in Step \ref{step2} is to treat  
the singular locus $Y_{\sing}$ and the indeterminacy locus $B$. 
The key point in the proof is that the indeterminacy locus $B$ is 
automatically killed and the singular locus $Y_{\sing}$ is also 
killed by the zero locus of the section $\widetilde{s}$. 
\end{proof}

As we mentioned in Remark \ref{main-rem}, 
by using the above method, 
Conjecture \ref{image-conj} can be solved 
if $Y$ admits a good minimal model $Y \dashrightarrow Y_{\min}$. 
The main idea is to kill 
the singular locus $Z_{\sing}$ and 
the non-smooth locus of a morphism $f:Y_{\min} \to Z$ 
by the zero locus of the section $\widetilde{s}$, 
where $Z$ is the canonical model of $Y_{\min}$.
For this purpose, the following lemma, 
which can be seen as a generalization of Claim \ref{claim-1} and Claim \ref{claim-2}, 
plays an important role, 
but we omit the detail. 

\begin{lemm}\label{key}
Let $X$, $Y$, and $Z$ be complex manifolds. 
For morphisms $\phi: X \to Y$ and $f:Y \to Z$, 
we assume that $\psi:=f \circ \phi: X \to Z$ is a smooth morphism at $p \in X$. 
Further let $g$ be a K\"ahler metric of $X$ with 
the semi-positive holomorphic sectional curvature $H_g$. 
We put $m:=\dim Z$. 
We consider the induced metric $\Lambda^m g^\vee$ on $\Lambda^m \Omega_{X}$. 
%and the injective linear map
%$$\Lambda^md\phi^*: 
%\phi^{*} K_{Y,p} 
%\xrightarrow{\quad \Lambda^md\phi^* \quad} 
%\Lambda^m \Omega_{X,p} \text{ at  $p \in X$. }
%$$
Then, for any non-zero vector vector  
$$
b  \in \Image (\Lambda^md\phi^*: 
\phi^{*} K_{Y,p} 
\xrightarrow{\quad \Lambda^md\phi^* \quad} 
\Lambda^m \Omega_{X,p}) \subset \Lambda^m \Omega_{X,p} \text{ at } p,  
$$
there exists a tangent vector $v \in T_{X,p}$ with the following properties\,$:$
$$\text{
$\bullet$ $d \psi_{*}(v) \not = 0$ in $T_{Z, \psi(p)}$. \quad \quad 
$\bullet$  $\Rur{\Lambda^m g^\vee}{v}{b} \leq 0$. 
} 
$$
\\
Moreover let $(A,H)$ be a smooth hermitian metric on $Z$ with  the positive curvature $\sqrt{-1}\Theta_H>0$. 
Then, 
for any non-zero vector 
\begin{align*}
b \in \Image \big( S^{\ell}( \Lambda^m d {\phi}^{*}) \otimes id : 
{\phi}^{*} (K_{Y,p}^{\otimes \ell}) \otimes\psi^{*}A_p^\vee
\xrightarrow {\quad S^{\ell}( \Lambda^m d {\phi}^{*}) \otimes id \quad }
S^{\ell}(\Lambda^m \Omega_{X,p}) \otimes \psi^{*}A_p^\vee
\big),  
\end{align*}
there exists a tangent vector $v \in T_{X,p}$ such that 
$$\Rur{S^{\ell}(\Lambda^m g^\vee) \otimes \psi^* H^\vee}{v}{b} < 0. 
$$
\end{lemm}

At the end of this paper, 
we  prove Corollary \ref{main-cor}.

\begin{proof}[{Proof of Corollary \ref{main-cor}}]
For a compact K\"ahler manifold 
with semi-positive holomorphic sectional curvature $H_g$, 
we can show that 
$X$ admits a finite \'etale cover by a complex torus or 
$K_{X}$ is not pseudo-effective. 
Indeed, when the holomorphic sectional curvature is identically zero, 
then $X$ admits a finite \'etale cover by a complex torus  
(see \cite[Proposition 2.2]{HLW16}, \cite{Ber66}, \cite{Igu54}). 
When it is not identically zero, 
we consider the scalar curvature $S$ of the K\"ahler metric $g$. 
Then we have 
$$
\int_{X} c_1(K_X) \wedge \omega_g^{n-1} = - \frac{1}{\pi n}\int_{X} S \, \omega_g^n, 
$$
where $\omega_g$ is the K\"ahler form associated to $g$.
The value of $S$ at a point $p \in X$ 
can be written as the integral of 
the holomorphic sectional curvature over 
the projective space $\mathbb{P}(T_{X,p}^\vee)$ (see \cite{Ber66}). 
Therefore the right hand side is negative by the assumption that 
$H_{g}([v]) > 0$ holds for some tangent vector $v$. 
In particular, the canonical bundle $K_X$ is not pseudo-effective.

We consider a compact K\"ahler surface $X$
such that the holomorphic sectional curvature is not identically zero. 
Then, by the above argument, 
we can see that $K_X$ is not pseudo-effective.
It is known that a compact complex surface such that $K_X$ is not pseudo-effective 
is a rational surface, or a minimal surface of class V\hspace{-.1em}I\hspace{-.1em}I, 
or a ruled surface over a curve of genus $\geq 1$ 
by the classification of compact complex surfaces. 
However a minimal surface of class V\hspace{-.1em}I\hspace{-.1em}I is not K\"ahler, 
and thus we can conclude that $X$ is rationally connected or a ruled surface over 
a curve of genus $\geq 1$. 
In the case where $X$ is a ruled surface, 
the genus of the base is less than or equal to one by Theorem \ref{main-thm}. 
Therefore the base of a ruled surface with semi-positive holomorphic sectional curvature 
is an elliptic curve. 
\end{proof}

%\newpage
%%%%%%%%%%%%%%%%%%%%%%

\end{document}